\documentclass{amsart}
\usepackage{amsmath,amssymb,amscd}
 
\usepackage{color}
\usepackage[usenames,dvipsnames,svgnames,table]{xcolor}

\renewcommand{\AA}{\mathbb A}
\newcommand{\CC}{\mathbb C}
\newcommand{\NN}{\mathbb N}
\newcommand{\QQ}{\mathbb Q}
\newcommand{\ZZ}{\mathbb Z}

\newcommand{\Gm}{{\mathbb G}_m}

\newcommand{\cO}{\mathcal O}

\newtheorem{thm}{Theorem}[section]
\newtheorem{lem}[thm]{Lemma}
\newtheorem{prop}[thm]{Proposition}
\newtheorem{question}[thm]{Question}

\theoremstyle{remark}
\newtheorem{Rk}[thm]{Remark}
\newtheorem{example}[thm]{Example}

\theoremstyle{definition}
\newtheorem{Def}[thm]{Definition}

\begin{document}

\title{Exponential-polynomial equations and dynamical return sets}
\author{Thomas Scanlon }
\thanks{This collaboration was made possible by NSF grant FRG DMS-0854839.  The first author was partially
supported by NSF grant DMS-1001556, and the second author by JSPS Grants-in-Aid 23740033.}
\address{University of California, Berkeley \\
Department of Mathematics \\
Evans Hall \\
Berkeley, CA 94720-3840 \\
USA}
\email{scanlon@math.berkeley.edu}

\author{Yu Yasufuku}
\address{Nihon University, College of Science and Technology\\
Department of Mathemtaics\\
1-8-14 Kanda-Surugadai\\
Chiyoda, Tokyo 101-8308\\
JAPAN}
\email{yasufuku@math.cst.nihon-u.ac.jp}

\maketitle

\begin{abstract}
We show that for each finite sequence of algebraic
integers $\alpha_1,\ldots,\alpha_n$ and polynomials
$P_1(x_1,\ldots,x_n;y_1,\ldots,y_n), \ldots, P_r(x_1,\ldots,x_n;y_1,\ldots,y_n)$
with algebraic integer coefficients, there are a natural number $N$,
$n$ commuting endomorphisms $\Phi_i:\Gm^N \to \Gm^N$ of the $N^\text{th}$
Cartesian power of the multiplicative group, a point $P \in \Gm^N(\QQ)$, and
an algebraic subgroup $G \leq \Gm^N$ so that the return set
$\{ (\ell_1,\ldots,\ell_n) \in \NN^n ~:~ \Phi_1^{\circ \ell_1} \circ\cdots \circ \Phi_n^{\circ \ell_n}(P) \in
G(\QQ) \}$ is identical to the set of solutions to the given exponential-polynomial
equation: $\{ (\ell_1,\ldots,\ell_n) \in \NN^n ~:~ P_1(\ell_1,\ldots,\ell_n;\alpha_1^{\ell_1},
\ldots,\alpha_n^{\ell_n}) = \cdots = P_r(\ell_1,\ldots,\ell_n;\alpha_1^{\ell_1},\ldots,\alpha_n^{\ell_n})
= 0 \}$.
\end{abstract}

\section{Introduction}

Motivated by the conclusion of Faltings' Theorem on rational points
on subvarieties of abelian varieties, Ghioca, Tucker and Zieve posed
the following question (Question 1.6 of~\cite{GTZ1})
about return sets for
finite rank algebraic dynamical systems.

\begin{question}
\label{qgtz}
Let $X$ be a variety defined over $\CC$, let $V$ be a closed subvariety of $X$,
let $S$ be a finitely   generated commutative subsemigroup of $\operatorname{End} X$,
and let $\alpha \in X(\CC)$. Do the following hold?
\begin{itemize}
\item[(a)] The intersection $V(\CC) \cap \cO_S(\alpha)$ can be written as $\cO_T (\alpha)$
where $T$ is the union of finitely many cosets of subsemigroups of $S$.

\item[(b)] For any choice of generators $\Phi_1,\ldots,\Phi_r$  of $S$, let $Z$ be the set of
tuples $(n_1, \ldots , n_r)\in \NN^r$
for which $\Phi_1^{\circ n_1} \circ \cdots \circ \Phi_r^{\circ n_r}(\alpha) \in V(\CC)$, where $\Phi_i^{\circ n_i}$ is the $n_i$-fold composition of $\Phi_i$;  then $Z$
is the union of finitely many sets of the form $z_i + (G_i \cap \NN^r)$,  where
each $G_i$ is a subgroup of $\ZZ^r$
and each $z_i \in \NN^r$.
\end{itemize}
\end{question}

There are some obvious cases in which Question~\ref{qgtz} has a negative
answer.  For example, if $X = \AA^m$ is the affine $m$-space, $\Phi_i:\AA^m \to \AA^m$ is
given by $(x_1,\ldots,x_n) \mapsto (x_1,\ldots,x_{i-1},x_i+1,x_{i+1},\ldots,x_n)$,
$\alpha = (0,\ldots,0)$ is the origin, and $V \subseteq \AA^n$ is any variety, then
$Z = V(\CC) \cap \NN^n$.  As is well-known, the set of natural number points on
a variety may be very complicated.  For another example, consider the case of
$X = \AA^2$, $\Psi_1(x,y) = (2x,y)$, $\Psi_2(x,y) = (x,y+1)$, $\alpha = (1,0)$,
and $Y = \Delta_{\AA^1} = \{  (x,y) ~:~ x = y \}$.  Then $Z = \{ (m,2^m) ~:~ m \in \NN \}$.

With these examples in mind, one might seek geometric conditions on the algebraic
dynamical system $(X,S)$ for which a positive answer to Question~\ref{qgtz} may be
expected.  In a companion paper~\cite{GTZ2}, the same authors
specialized Question~\ref{qgtz} to the case that $X = \Gm^g$ is a power of the multiplicative
group and $S$ is a semigroup of algebraic group endomorphisms.  Under various hypotheses, for
example when the differential of each $\Phi_i$ at the origin is diagonalizable, they showed that
Question~\ref{qgtz} has a positive answer, but they constructed two examples for which the
return sets are infinite but may be represented as the natural number points on a quadratic curve.

In this note we show that such examples are far from anomalous and that, in fact,
every set which may be expressed as the natural number solutions of an
exponential-polynomial equation may be realized as the return set for an algebraic
dynamical system on some power of the multiplicative group.  We proceed by running the
by-now-standard Skolem-Mahler-Lech-Chabauty argument in reverse.   That is, we
start with some easy linear algebraic calculations showing that if $R$ is any
commutative ring with no $\ZZ$-torsion, then every set of natural numbers solutions to a system of exponential-polynomial
equations over $R$ may be realized as the return set for a \emph{linear} dynamical system
over $R$.  We pull this result down from rings of integers in number fields to $\ZZ$ to show that
every set of solutions to a system of exponential-polynomial equations over $R$ may be realized
as the return set of a linear dynamical system over $\ZZ$.  Exponentiating this last linear dynamical system
we obtain the desired algebraic dynamical system on an algebraic torus.

\section{Conventions and statement of main theorem}

We include $0$ in the set $\NN$ of natural numbers.  Our fundamental object of study is the return set for an algebraic dynamical system.

\begin{Def}
Given a set $X$,
a finite sequence $\Phi_1, \ldots, \Phi_n$ of self-maps $\Phi_i:X \to X$, a point $a \in X$, and a subset $Y \subseteq X$, we define the
\emph{return set}  to be
$$
E(a,\Phi_1, \ldots, \Phi_n,Y) := \{ (\ell_1, \ldots, \ell_n) \in \NN^n  ~:~ \Phi_1^{\circ \ell_1} \circ \cdots \circ \Phi_n^{\circ \ell_n}(a) \in Y \}
$$
\end{Def}

\begin{Rk}
We shall abuse notation somewhat in the case of algebraic dynamical systems.  That is, if $X$ is a scheme over the ring $R$,
$\Phi_1,\ldots,\Phi_n$ is a sequence of commuting regular self-maps $\Phi_i:X \to X$, $Y \subseteq X$ is a subscheme and $a \in X(R)$
is an $R$-valued point of $X$, then we write $E(a,\Phi_1, \ldots, \Phi_n,Y)$ for $E(a,\Phi_1^R,\ldots,\Phi_n^R,Y(R))$.
\end{Rk}

Our main theorem is that the class of return sets for finitely generated commutative semigroups of algebraic group endomorphisms
of algebraic tori coincides with the class of exponential-polynomial sets.

\begin{Def}
\label{defexppoly}
Let $R$ be a commutative ring and $n$ a natural number.  An \emph{$R$-exponential-polynomial function of
$n$-variables} is a function $f:\NN^n \to R$ of the form
$(\ell_1,\ldots,\ell_n) \mapsto P(\ell_1,\ldots,\ell_n;\alpha_{1}^{\ell_1},\ldots,\alpha_m^{\ell_1},\ldots,\alpha_1^{\ell_n},\ldots,
\alpha_m^{\ell_n})$ for some polynomial $P$ in $n (m+1)$ variables over $R$ and elements $\alpha_1,\ldots,\alpha_m \in R$.
By an \emph{$R$-exponential-polynomial set} we mean a finite intersection of subsets of  $\NN^n$ defined by the vanishing of an $R$-exponential-polynomial function.
\end{Def}

With these definitions in place we may express our main theorem.

\begin{thm}
\label{main}
Let $\cO$ be the ring of all algebraic integers and let $Z \subseteq \NN^n$ be an $\cO$-exponential-polynomial set.  Then there are
an algebraic torus $X$ over $\QQ$, an $n$-tuple of commuting endomorphisms $\Phi_i:X \to X$, a point $P \in X(\QQ)$, and
an algebraic subgroup $Y \leq X$  for which $Z = E(P,\Phi_1,\ldots, \Phi_n,Y)$.
\end{thm}

\section{Some basic lemmata on exponential-polynomials}

For the remainder of this note, $R$ denotes a commutative ring with no $\ZZ$-torsion.  We write $R_\QQ := R \otimes \QQ$ and regard $R$ as a subring of $R_\QQ$.

We shall encode general exponential-polynomial sets by representing their defining equations as
linear relations amongst basic generalized monomials, but for the sake of concreteness, we  regard
exponential polynomials as functions.

\begin{Def}
For $k \in \NN$ and $a \in R$ we define $\binom{a}{k} := \frac{1}{k!} \prod_{i=0}^{k-1} (a-i) \in R_\QQ$,  where $\binom a0 := 1$ as usual.
If ${\mathbf a} : = (a_1,\ldots,a_n) \in R^n$ is an $n$-tuple of elements of $R$ and ${\mathbf k} = (k_1,\ldots,k_n) \in \NN^n$ is an $n$-tuple of natural numbers,
then $\binom{\mathbf a}{\mathbf k} := \prod_{i=1}^n \binom{a_i}{k_i}$ and ${\mathbf a}^{\mathbf k} := \prod_{i=1}^n a_i^{k_i}$.  By a
\emph{basic exponential multinomial over $R$} we mean an $R_\QQ$-exponential polynomial
of the form $\binom{\mathbf x}{\mathbf k} {\boldsymbol \lambda}^{\mathbf x}$ for some
${\mathbf k} \in \NN^n$ and ${\boldsymbol \lambda} \in R^n$ where ${\mathbf x} = (x_1,\ldots,x_n)$ is the $n$-tuple of standard indeterminates in the
polynomial ring $R[x_1,\ldots,x_n]$.
\end{Def}

\begin{lem}\label{newmultipolybas}
Every element of $R[x_1,\ldots, x_n]$ can be expressed as an $R$-linear combination of the set $\{\binom {\mathbf x}{\mathbf k} ~:~ \mathbf k\in \NN^n\}$.
\end{lem}

\begin{proof}
For each $i$, we prove by induction that $x_i^k$ for any $k$ is an $\ZZ$-linear combination of the set $\{\binom {x_i}j ~:~ j\in \NN\}$. Indeed $x_i^0 = 1 = \binom{x_i}0$.  More generally, $x_i^k - k! \binom{x_i}k$ is a polynomial with $\ZZ$-coefficients of degree less than $k$.  So this completes the induction.  By expansion, we see that any monomial $\mathbf x^{\mathbf j}$ is a $\ZZ$-linear combination of the set $\{\binom{\mathbf{x}}{\mathbf k} ~:~ \mathbf k \in \NN^n\}$.  By tensoring with $R$, the result follows.
\end{proof}

\begin{lem}
\label{multiexppolybas}
Every $R$-exponential-polynomial function may be expressed as a
finite $R$-linear combination of basic exponential multinomials.
\end{lem}
\begin{proof}
Using the laws of exponents, it is easy to see that every $R$-exponential-polynomial function may be expressed as a finite $R$-linear
combination of exponential polynomial functions of the form $ {\boldsymbol \lambda}^{\mathbf x} {\mathbf x}^{\mathbf k} $.   By
Lemma~\ref{newmultipolybas}, the monomials ${\mathbf x}^{\mathbf k}$ may be expressed as $\ZZ$-linear combinations of
basic multinomials.  Distributing the product of the exponential term over the sum, we conclude.
\end{proof}

\section{Some linear algebra}
In this section we carry out some basic linear algebraic computations in the service of our main theorem.

Some notation is in order.

\begin{Def}
For any natural number $n$ and $i \leq n$, we denote by $e_{i,n}$ (written as $e_i$ if $n$ is understood) the column vector
whose $i^\text{th}$ entry is $1$ and all of whose other entries are $0$.  That is, $e_{1,n},\ldots,e_{n,n}$ is the
 standard basis of $\ZZ^n$.  The linear map $J_n$ is defined by $J_n e_{i,n}  = e_{i+1,n}$ for $i < n$ and $J_n e_{n,n} = 0$.
That is, considered as an $n \times n$ matrix $J_n$ is the element of $M_{n \times n}(\ZZ)$ with $1$s along the
subdiagonal
and $0$s in every
other entry.  If $n$ is understood or otherwise immaterial, we write $J$ for $J_n$.   We write $I_n$ for the identity matrix in
$M_{n \times n}$ and again write $I$ if $n$ is understood.   For any ring $R$ as there is a unique map $\ZZ \to R$, we regard $J_n$ and $I_n$ as
elements of $M_{n \times n}(R)$ and $e_{i,n}$ as an element of $R^n$.
\end{Def}

\begin{lem}
\label{injdot}
For any $n$-tuple ${\mathbf j} = (j_1,\ldots,j_n)$ of natural numbers, there is another
$n$-tuple ${\mathbf M} = (M_1,\ldots,M_n)$ of natural numbers having the property that the only
$n$-tuple ${\mathbf k} = (k_1,\ldots,k_n)$ of natural numbers satisfying ${\mathbf k} \cdot {\mathbf M}
 := \sum_{i=1}^n k_i M_i = {\mathbf j} \cdot {\mathbf M}$ is ${\mathbf k} = {\mathbf j}$.
\end{lem}
\begin{proof}
Let $p_1,\ldots,p_n$ be a sequence of distinct primes for which $j_i < p_i$ for each $i \leq n$.  Set
$M_i := \prod_{\ell \neq i} p_\ell$.  If ${\mathbf k} \cdot {\mathbf M} = {\mathbf j} \cdot {\mathbf M}$, then
$k_i M_i \equiv j_i M_i \pmod{p_i}$ for each $i \leq n$.  As $M_i$ is a product of primes distinct from the prime
$p_i$ we conclude that $M_i$ is invertible modulo $p_i$ so that $k_i \equiv j_i \pmod{p_i}$.  Thus, we may write $k_i = j_i + \epsilon_i p_i$
for some integer $\epsilon_i$.  As $j_i < p_i$ and $k_i \geq 0$, we conclude that $0 \leq k_i = j_i + \epsilon_i p_i < (1+\epsilon_i) p_i$ so that
$\epsilon_i \geq 0$.  We then have
$${\mathbf j} \cdot {\mathbf M} = {\mathbf k} \cdot {\mathbf M} = \sum_{i=1}^n (j_i + \epsilon_i p_i) M_i =
{\mathbf j} \cdot {\mathbf M} + \left(\sum_{i=1}^n \epsilon_i\right)  \left(\prod_{\ell =1}^n p_\ell\right).
$$Thus, $0 = \sum_{i=1}^n \epsilon_i$.  As each $\epsilon_i$
is nonnegative, we conclude that they are all equal to zero.  That is, ${\mathbf j} = {\mathbf k}$.
\end{proof}

\begin{prop}
\label{monoencode}
For any natural number $n$, $n$-tuple ${\boldsymbol \lambda} \in R^n$ of elements of $R$ and  $n$-tuple
${\mathbf j} \in \NN^n$ of natural numbers, there are some finite rank free $R$-module $F$,
an $R$-linear map $\pi:F \to R$,
an element $v$,  and an $n$-tuple $\psi_1, \ldots, \psi_n$ of commuting endomorphisms of $F$
so that for all $n$-tuples ${\boldsymbol \ell}  \in \NN^n$ of natural numbers one has
$$
\pi \circ \psi_1^{\circ \ell_1} \circ \cdots \circ \psi_n^{\circ \ell_n} (v) =  {\boldsymbol \lambda}^{\boldsymbol \ell}  \binom{\boldsymbol \ell}{\mathbf j}
$$
\end{prop}
\begin{proof}
Let ${\mathbf M} = (M_1,\ldots,M_n) \in \NN^n$ be the sequence of natural numbers provided by Lemma~\ref{injdot}.  Let
$N := {\mathbf M} \cdot {\mathbf j} + 1$ and $F := R^N$, and set $\psi_i := \lambda_i (I + J^{M_i})$.   As $\{\psi_1, \ldots, \psi_n \} \subseteq R[J]$,
the subring of $M_{N \times N}(R)$
they generate is commutative.  Using the usual binomial expansions, one computes immediately that for any $(\ell_1,\ldots,\ell_n) \in \NN^n$,
one has

\begin{eqnarray*}
\left(\psi_1^{\circ \ell_1} \circ \cdots \circ \psi_n^{\circ \ell_n}\right) (e_{1,N})  
& = & \left(\prod_{i=1}^n  (\lambda_i (I + J^{M_i}))^{\ell_i} \right) (e_{1,N}) \\
 & = & \sum_{ {\mathbf k} \in \NN^n}  {\boldsymbol \lambda}^{\boldsymbol \ell} \binom{\boldsymbol \ell}{\mathbf k}  J^{{\mathbf k} \cdot {\mathbf M} } (e_{1,N})  \\
 & = & \sum_{m = 0}^{N}  \left(\sum_{ {\mathbf k } \cdot {\mathbf M} = m }  {\boldsymbol \lambda}^{\boldsymbol \ell}  \binom{\boldsymbol \ell}{\mathbf k}\right)  e_{m+1,N}
\end{eqnarray*}

As the only solution to ${\mathbf k} \cdot {\mathbf M} = N-1$ is given by ${\mathbf k} = {\mathbf j}$, we conclude that the coefficient of $e_{N,N}$ in
$\psi_1^{\circ \ell_1} \circ \cdots \circ \psi_n^{\circ \ell_n} (v)$ is ${\boldsymbol \lambda}^{\boldsymbol \ell}  \binom{\boldsymbol \ell}{\mathbf j}$, where $v := e_{1,N}$. Let $\pi:F \to R$ be the projection onto the $N^\text{th}$ co-ordinate.
\end{proof}

\begin{Def}
For ${\mathbf j} = (j_1,\ldots,j_n)$ and ${\boldsymbol \lambda} = (\lambda_1,\ldots,\lambda_n) \in R^n$ as in Proposition~\ref{monoencode}, we let
$(F_{{\mathbf j},{\boldsymbol \lambda}}, \psi_{1;{\mathbf j},{\boldsymbol \lambda}}, \ldots, \psi_{n;{\mathbf j},{\boldsymbol \lambda}}, v_{{\mathbf j},{\boldsymbol \lambda}},
\pi_{{\mathbf j},{\boldsymbol \lambda}})$
be the module, commuting linear maps, vector and projection map obtained in Proposition~\ref{monoencode}.
\end{Def}

\begin{thm}
\label{linearreturn}
For any commutative ring $R$ with no $\ZZ$-torsion and $R$-exponential-polynomial set $Z \subseteq \NN^n$ there are a
finite-rank free $R$-module $F$, $n$-tuple $\psi_1,\ldots, \psi_n$ of commuting $R$-module endomorphisms on $F$, point $a \in F$, and submodule $S \leq F$
for which $Z = E(a,\psi_1,\ldots, \psi_n, S)$.  Moreover, $S$ may be taken to be the kernel of an $R$-linear map $\theta: F\to Q$, where $Q$ is also a finite-rank free $R$-module.
\end{thm}
\begin{proof}
It suffices to show that the zero set of a single $R$-exponential-polynomial may be encoded as a return set.  Indeed, if $(F_i, \psi_{i,1},\ldots, \psi_{i,n},a_i,S_i)$
has return set $Z_i$ for $1 \leq i \leq m$, then
\[
\left(\bigoplus_{i=1}^m F_i,\,\,\, \bigoplus_{i=1}^m \psi_{i,1},\,\,\ldots,\,\,\bigoplus_{i=1}^m \psi_{i,n},\,\,\, (a_1,\ldots,a_m),\,\,\, \bigoplus_{i=1}^m S_i\right)
\]
has return set $\bigcap_{i=1}^m Z_i$.

Let  $Z$ be the zero set of an $R$-exponential-polynomial function $f({\mathbf x})$.   By Lemma~\ref{multiexppolybas} we may
write $f = \sum r_{{\mathbf j}, {\boldsymbol \lambda}} {\boldsymbol \lambda}^{\mathbf x} \binom{ \mathbf x}{\mathbf j}$ where $r_{{\mathbf j}, {\boldsymbol \lambda}}\in R$ and the sum is taken over a finite set $A$ of
pairs $({\mathbf j}, {\boldsymbol \lambda})$ where ${\mathbf j} \in \NN^n$ and ${\boldsymbol \lambda} \in R^n$.

Let $F := \bigoplus F_{{\mathbf j},{\boldsymbol \lambda}}$, $\psi_i :=  \bigoplus \psi_{i; {\mathbf j}, {\boldsymbol \lambda}}$
for $i \leq n$, $a = \oplus v_{ {\mathbf j}, {\boldsymbol \lambda} }$, $Q = R$, and $S$ to be the kernel of the composite $\theta$ of
$\bigoplus r_{{\mathbf j},{\boldsymbol \lambda}} \pi_{{\mathbf j},{\boldsymbol \lambda}}$ and the sum map, where the direct sum and the sum are over $({\mathbf j}, {\boldsymbol \lambda}) \in A$.

From Proposition~\ref{monoencode} for any ${\boldsymbol \ell} \in \NN^n$ we have $\theta \circ  \psi_1^{\circ \ell_1} \circ \cdots \circ
\psi_n^{\circ \ell_n} (a) = f({\boldsymbol \ell})$.  Thus, $Z = E(a,\psi_1,\ldots, \psi_n, S)$ as claimed.
\end{proof}

\section{Return sets on tori}
In this section we deduce Theorem~\ref{main} from the linear algebraic Theorem~\ref{linearreturn}.

Throughout this section we denote by $\cO$ the ring of all algebraic integers.

Let us note first that every exponential-polynomial set over the algebraic integers may be realized by a linear dynamical system over $\ZZ$.

\begin{prop}
\label{linearZ}
If $Z \subseteq \NN^n$ is a $\cO$-exponential-polynomial set, then
there are natural numbers $r$ and $s$,  a sequence $\phi_1, \ldots, \phi_n$ of commuting $r \times r$-matrices over $\ZZ$, a vector $\mathbf a \in \ZZ^r$, and
a $\ZZ$-module map $L:\ZZ^r \to \ZZ^s$ so that if $T := \ker L$, then $Z = \left\{ {\boldsymbol \ell} \in \NN^n ~:~ \phi_1^{\circ \ell_1} \cdots  \phi_n^{\circ \ell_n} (\mathbf a) \in T  \right\}$.
\end{prop}
\begin{proof}
Let $R$ be the subring of $\cO$ generated by the bases of the exponents appearing in the expressions of
some finite list of $\cO$-exponential-polynomial functions whose zero set is equal to $Z$.   As each such number is
integral over $\ZZ$, $R$ is free of finite-rank as a $\ZZ$-module.  Let $(F, \psi_1,\ldots, \psi_n,a,S)$
be given by Theorem~\ref{linearreturn} for
$R$ and $Z$ and let $\theta:F \to Q$ be an $R$-linear map with $\ker \theta = S$.   Then $F$ is also a finite rank free $\ZZ$-module and all of
the listed maps are $\ZZ$-linear.  Choosing a basis, we may identify $F$ with $\ZZ^r$ and each $\psi_i$ with some $r \times r$ matrix $\phi_i$.  Likewise,
fixing a $\ZZ$-basis for $Q$, we may regard $\theta$ as an $s \times r$ matrix.   As the dynamical systems $(F,\psi_1,\ldots, \psi_n)$ and $(\ZZ^r,\phi_1,\ldots,\phi_n)$
(after a choice of basis) are identical as are the initial points and target sets, their return sets are identical.
\end{proof}

Finally, let us finish the proof of the main theorem.

\begin{proof}[Proof of Theorem~\ref{main}]
Let $Z \subseteq \NN^n$ be any $\cO$-exponential-polynomial set.   Let $r$, $s$, $L$, $T$,   $\mathbf a$, and $\phi_1, \ldots, \phi_n$ be given by
Proposition~\ref{linearZ}.   Since $\operatorname{End}(\Gm) = \ZZ$, we may identify $\operatorname{Hom}(\Gm^r,\Gm^s)$ with
$M_{s \times r}(\ZZ)$ and $\operatorname{End}(\Gm^r)$ with $M_{r \times r}(\ZZ)$.  Let $\Phi_i:X \to X$ be the endomorphism of $\Gm^r$
corresponding to $\phi_i$ under this identification and let $Y$ be the kernel of the map corresponding to $L$.  Let $P := (2^{a_1}, \ldots, 2^{a_r})$
where the $a_i$s are the components of $\mathbf a$.  Since $2$ has infinite order, $E(P,\Phi_1,\ldots, \Phi_n,Y) = E(\mathbf a,\phi_1,\ldots,\phi_n, T) = Z$.
\end{proof}

\section{Concluding remarks}

We end this note with a few observations, an explicit example, and some open questions.

\begin{Rk}
If $X$ is a semiabelian variety over a field $K$ of characteristic zero, $\Phi_1, \ldots, \Phi_n$ is a finite sequence of
commuting endomorphisms of $X$, $Y \subseteq X$ is a subvariety, and $a \in X(K)$ is any point, then the
return set $E(a, \Phi_1, \ldots, \Phi_n, Y)$ is necessarily an $\cO$-exponential-polynomial set.  Indeed, this result follows
from the Mordell-Lang conjecture (or theorem of Faltings and Vojta) and the Skolem-Mahler-Lech-Chabauty method and is implicit
in~\cite{GTZ2}.   Our Theorem~\ref{main} generalizes immediately to the case
that $X$ is taken to be a power of a semiabelian variety instead of $\Gm$.   Thus, Theorem~\ref{main} may be read as saying that
the class of return sets for actions of finitely generated commutative monoids on semiabelian varieties over fields of
characteristic zero is \emph{precisely} the class of $\cO$-exponential-polynomial sets.
\end{Rk}

\begin{Rk}
We have stated Theorem~\ref{main} as an identity of point sets, but as the reader will see from from the proof, we actually
convert a system of defining equations for an $\cO$-exponential-polynomial system into an algebraic dynamical system with a
fixed starting point and target set so that the problem of membership in the return set is reducible to the corresponding
problem of solving the given exponential-polynomial equations.   That is, these problems are computationally equivalent.
\end{Rk}

\begin{example}\label{example}
Our method of construction is effective.   For example, let $f(\ell_1,\ell_2) = (1+\sqrt 2)^{\ell_1} \ell_1 \ell_2 - 21 \ell_2^2 - 5\sqrt 2 \ell_1$, whose zeroes include $(3,1)$.  Using binomials, $f(\ell_1,\ell_2) = (1+\sqrt 2)^{\ell_1} \ell_1 \ell_2 - 42 \binom{\ell_2}2 - 21 \ell_2 - 5\sqrt 2 \ell_1$.  Let $R = \ZZ[\sqrt 2]$.  We can actually use $\mathbf M = (3,2)$ for each of the four terms, so $\psi_2$ is $I+J^2$ for each of the four blocks and $\psi_1$ is $(1+\sqrt 2) (I+J^3)$ for the first block and $(I+J^3)$ for the last three blocks.  The sizes of the blocks are $\mathbf j\cdot \mathbf M +1$, so they are $6$, $5$, $3$, and $4$ in order.  Letting $\mathbf a$ be the vector which is $1$ in $x_1$, $x_7$, $x_{12}$, and $x_{15}$ and $0$ everywhere else, the condition that $\psi_1^{\circ \ell_1} \circ \psi_2^{\circ \ell_2}(\mathbf a)$ is in the set
defined by $x_6 - 42 x_{11} - 21 x_{14} - 5\sqrt 2 x_{18} = 0$ is precisely $f(\ell_1,\ell_2)$.  Finally, let $x_i = y_i + z_i\sqrt{2}$ and exponentiate: the first coordinate of $\psi_1(x_1,\ldots,x_{18})$ is  $(1+\sqrt 2)x_1 = (y_1+2z_1) + (y_1+ z_1)\sqrt 2$, so the first two coordinates of $\Phi_1(Y_1,Z_1,\ldots, Y_{18},Z_{18})$ are $Y_1 Z_1^2$ and $Y_1 Z_1$ and we continue in the same manner to construct $\Phi_1$ and $\Phi_2$.  Here, $P$ is the point which is $2$ at $Y_1$, $Y_7$, $Y_{12}$, and $Y_{15}$ and $1$ everywhere else, and the subgroup $Y$ is defined by $Y_6 = Y_{11}^{42} Y_{14}^{21} Z_{18}^{10}$ and $Z_6 = Y_{18}^5$.
\end{example}

\begin{Rk}
As is clear from Example~\ref{example}, our construction in proving Theorem~\ref{main} does not optimize the dimension of the algebraic torus on which our dynamical system acts. In fact, any $\ZZ$-linear function $r_1 \ell_1 + r_2 \ell_2$ can be achieved by a $2\times 2$ block: $\psi_1 = I+ r_1 J$, $\psi_2 = I+ r_2 J$.  It remains an open question to determine the minimum dimension of $\Gm^N$ on which we can generate all of the exponential-polynomial functions of a fixed degree, where the degree of an exponential-polynomial function $P(\ell_1,\ldots,\ell_n;\alpha_{1}^{\ell_1},\ldots,\alpha_m^{\ell_1},\ldots,\alpha_1^{\ell_n},\ldots,
\alpha_m^{\ell_n})$ is defined to be the degree in the first $n$ variables.
\end{Rk}

\begin{Rk}
It follows from our Theorem~\ref{main} and the work of Davis, Putnam and Robinson~\cite{DPR} on the representability of
recursively enumerable sets as exponential diophantine sets, that many natural questions about algebraic dynamics are
undecidable.   For example, there is no algorithm which takes as input a tuple of the form $(N,\Phi_1,\ldots,\Phi_{n+1},m,a,T)$
where $n$, $m$, and $N$ are natural numbers, $\Phi_i:\Gm^N \to \Gm^N$ are commuting endomorphisms, $T \leq \Gm^N$, $a \in \Gm^N(\QQ)$
and answers correctly whether or not there is an $n$-tuple $(\ell_1,\ldots,\ell_n) \in \NN^n$ with
$\Phi_1^{\circ \ell_1} \circ \cdots \circ \Phi_n^{\circ \ell_n} \circ \Phi_{n+1}^{\circ m} (a) \in T(\QQ)$.
\end{Rk}


\begin{thebibliography}{99}
\bibitem{DPR}   Martin Davis, Hilary Putnam,  and Julia Robinson, The decision problem for exponential {D}iophantine equations, \emph{Ann. of Math.}, {\bf 74} no. 3 (1961),   425 -- 436.
\bibitem{GTZ1}  Dragos Ghioca, Thomas J. Tucker, and Michael E. Zieve, Linear relations between polynomial orbits, \emph{Duke Math. J.}  {\bf 161} (2012), no. 7, 1379 -- 1410.
\bibitem{GTZ2}  Dragos Ghioca, Thomas J. Tucker,  and Michael E. Zieve, The Mordell-Lang question for endomorphisms of semiabelian varieties, \emph{J. Th\'{e}or. Nombres Bordeaux} {\bf 23} (2011), no. 3, 645 --666.

\end{thebibliography}
\end{document}